\documentclass[11pt,reqno]{amsart}
\usepackage{amsmath, amsfonts, amsthm, amssymb, color, cite, soul}
\usepackage{mathrsfs}
\usepackage{dsfont}
\usepackage{upref}
\usepackage{indentfirst}
\usepackage{appendix}
\usepackage{enumerate}
\usepackage[bookmarks,colorlinks,citecolor=blue,linkcolor=red]{hyperref}
\usepackage{bm}
 \usepackage{graphicx}
\allowdisplaybreaks
\numberwithin{equation}{section}
\numberwithin{figure}{section}
\newtheorem{theorem}{Theorem}[section]

\newtheorem{remark}{Remark}[section]
\newtheorem{lemma}{Lemma}[section]

\author[Y.-Z. Chen]{Yazhou Chen}
\address{ College of Mathematics and Physics, Beijing University of
Chemical Technology, Beijing 100029, China}
\email{chenyz@mail.buct.edu.cn}

\author[Q.-L. He]{Qiaolin He}
\address{School of Mathematics, Sichuan University, Sichuan 610065, China}
\email{qlhejenny@scu.edu.cn}

\author[B. Huang]{Bin Huang}
\address{College of Mathematics and Physics, Beijing University of
Chemical Technology, Beijing 100029, China}
\email{abinhuang@gmail.com}

\author[X.-D. Shi]{Xiaoding Shi}
\address{College of Mathematics and Physics, Beijing University of
Chemical Technology, Beijing 100029, China}
\email{shixd@mail.buct.edu.cn}

\title[Viscous Compressible Heat-conducting Navier-Stokes-Allen-Cahn System]
{The Cauchy Problem for Non-Isentropic Compressible Navier-Stokes/Allen-Cahn system with Degenerate Heat-Conductivity}



\date{\today}

\begin{document}
\begin{abstract}
The Cauchy problem for non-isentropic compressible Navier-Stokes/Allen-Cahn system with degenerate heat-conductivity $\kappa(\theta)=\tilde{\kappa}\theta^\beta$ in 1-d is discussed in this paper. This system is widely used to describe the motion of immiscible two-phase flow in numerical simulation. The wellposedness for strong solution of this problem is established with the $H^1$ initial data for density, temperature, velocity, and the $H^2$ initial data for phase field. The result shows that no discontinuity of the phase field, vacuum, shock wave, mass or heat concentration will be developed at any finite time in the whole space. From the hydrodynamic point of view, this means that no matter how complex the interaction between the hydrodynamic and phase-field effects, phase separation will not occur, but the phase transition is possible.
\end{abstract}

\maketitle
\indent
In this paper, the motion of the viscous compressible immiscible two-phase flow is considered,  which is described by the following Navier-Stokes/Allen-Cahn system 
\begin{equation}\label{original NSAC}
\left\{\begin{array}{llll}
\displaystyle \rho_{t}+\textrm{div}(\rho \mathbf{u})=0,\\
\displaystyle (\rho \mathbf{u})_{t}+\mathrm{div}\big(\rho \mathbf{u}\otimes \mathbf{u}\big)=\mathrm{div}\mathbb{T},
  \\
\displaystyle(\rho\phi)_{t}+\mathrm{div}\big(\rho\phi \mathbf{u}\big)=-\mu,\\
\displaystyle\rho\mu=\rho\frac{\partial f}{\partial \phi}-\mathrm{div}\big(\rho\frac{\partial f}{\partial \nabla\phi}\big),\\
\displaystyle(\rho E)_{t}+\mathrm{div}(\rho E\mathbf{u})=\mathrm{div}\big(\mathbb{T}\mathbf{u}+\kappa(\theta)\nabla \theta-\mu\frac{\partial f}{\partial \nabla\phi}\big),
\end{array}\right.
\end{equation}
where $ \mathbf{x}\in \mathbb{R}^N $ is the spatial variable, $t$  the time, $N$   the spatial dimension. The unknown functions $\rho(\mathbf{x},t)$, $\mathbf{u}(\mathbf{x},t)$,  $\theta(\mathbf{x},t)$, $\phi(\mathbf{x},t)$  denote the total density, the velocity,  the absolute temperature and the phase field of the immiscible two-phase flow respectively. $\mu(\mathbf{x},t)$ is the chemical potential, $f$   the  phase-phase interfacial free energy density,  here we consider its common form as following (see Heida-M$\mathrm{\acute{a}}$lek-Rajagopal \cite{HMR2012} and the references therein) 
 \begin{equation}\label{free energy density}
 f(\rho,\phi,\nabla\phi)\overset{\text{def}}{=}\frac{1}{4\epsilon}(1-\phi^2)^2+\frac{\epsilon}{2\rho}|\nabla \phi|^2,
\end{equation}
where $\epsilon>0$  the thickness of the interface between the phases.
The Cauchy stress-tensor $\mathbb{T}$ is represented by
\begin{equation}\label{T}
\mathbb{T}=2\nu\mathbb{D}(\mathbf{u})+\lambda(\mathrm{div}\mathbf{u})\mathbb{I}-p\mathbb{I}-\rho\nabla\phi\otimes\frac{\partial f}{\partial\nabla\phi}.
\end{equation}
where  $\mathbb{I}$ is the unit matrix, $\mathbb{D}\mathbf{u}$ is the so-called deformation tensor
\begin{equation}\label{Du}
  \mathbb{D}\mathbf{u}=\frac{1}{2}\big(\nabla \mathbf{u} +\nabla^{\top} \mathbf{u}\big),
\end{equation}
here the superscript $\top$  denotes the transpose and all vectors are column ones. The positive constants $\nu>0,\lambda>0$ are  viscosity coefficients, satisfying
\begin{equation}\label{nu}
 \nu>0,\ \ \lambda+\frac{2}{N}\nu\geq0.
\end{equation}
The  total energy density $\rho E$ is given by
\begin{equation}\label{total energy density}
 \rho E=\rho e+\rho f+\frac{1}{2}\rho\mathbf{u}^2,
\end{equation}
where $\rho e$ is the internal energy, $\frac{\rho\mathbf{u}^2}{2}$  the kinetic energy. The pressure
 $p$, entropy $s$ and $\rho$, $e$, $f$ obey the  second law of thermodynamics
\begin{equation}\label{second law of thermodynamics}
  ds=\frac{1}{\theta}\big(d(e+f)+pd(\frac{1}{\rho})\big),
\end{equation}
 which implies that
\begin{equation}\label{pressure}
p=\rho^2\frac{\partial (e+f)}{\partial \rho}+\theta\frac{\partial p}{\partial \theta}=\rho^2\frac{\partial e(\rho,\theta)}{\partial \rho}-\frac\epsilon2|\nabla\phi|^2+\theta\frac{\partial p}{\partial \theta}.
\end{equation}
Throughout this paper, we consider the ideal polytropic gas, that is,
\begin{equation}\label{pressure p}
p(\rho,\theta)=R\rho\theta-\frac\epsilon2|\nabla\phi|^2,
\end{equation}
and
 $e$ satisfies
\begin{equation}\label{e}
  e=c_v\theta+\mathrm{constant},
\end{equation}
where $c_v$ is the specific heat capacity. $\kappa(\theta)$ is the heat conductivity satisfying
\begin{equation}\label{kappa}
 \kappa(\theta)=\kappa_0\theta^\beta,
\end{equation}
with positive constants $\kappa_0>0$ and $\beta>0$.

Substituting \eqref{free energy density}, \eqref{T}, \eqref{total energy density}, \eqref{pressure p} and \eqref{e} into \eqref{original NSAC}, then \eqref{original NSAC} is simplified as
\begin{equation}\label{NSFAC3d}
\left\{\begin{array}{llll}
\displaystyle \rho_{t}+\textrm{div}(\rho \mathbf{u})=0,\\
\displaystyle \rho\mathbf{u}_{t}+\rho(\mathbf{u}\cdot\nabla)\mathbf{u}-2\nu\mathrm{div} \mathbb{D}\mathbf{u}-\lambda\nabla\mathrm{div}\mathbf{u}=-\mathrm{div}\big(\epsilon\nabla\phi\otimes\nabla\phi-\frac{\epsilon}{2}|\nabla\phi|^2+\theta\frac{\partial p}{\partial \theta}\big), \\
\displaystyle \rho\phi_{t}+\rho \mathbf{u}\cdot\nabla\phi=-\mu,\\
\displaystyle \rho\mu=\frac{\rho}{\epsilon}(\phi^3-\phi)-\epsilon \Delta\phi,\\
\displaystyle c_v\big(\rho\theta_{t}+\rho \mathbf{u}\cdot\nabla\theta\big)+\theta p_\theta\mathrm{div}\mathbf{u}-\mathrm{div}(\kappa(\theta)\nabla \theta)=2\nu|\mathbb{D}\mathbf{u}|^2+\lambda(\mathrm{div}\mathbf{u})^2+\mu^2.
\end{array}\right.
\end{equation}

In this paper, 
we study the well-posedness for the one-dimensional case of system \eqref{NSFAC3d}. Under the Lagrange coordinate transformation,  
 the Cauchy problem for system \eqref{NSFAC3d} in 1-D can be rewritten as
\begin{equation}\label{NSFAC-Lagrange}
\left\{\begin{array}{llll}
\displaystyle v_t-u_x=0,\\
\displaystyle u_t+(\frac{\theta}{v})_x=( \frac{u_{x}}{v})_x-\frac\epsilon2(\frac{\phi_x^2}{v^2})_x, \\
\displaystyle\theta_t+\frac{\theta}{v}u_x-(\frac{\theta^\beta\theta_x}{v})_x=\frac{u_x^2}{v}+v\mu^2,\\
\displaystyle \phi_t=-v\mu,\\
\displaystyle\mu=\frac{1}{\epsilon}(\phi^3-\phi)-\epsilon(\frac{\phi_x}{v})_x,
\end{array}\right.
\end{equation}
with initial condition
\begin{equation}\label{initial condition}
 (v,u,\theta,\phi)(x,0)=(v_0,u_0,\theta_0,\phi_0)(x),\ \ (v_0,u_0,\theta_0,|\phi_0|)(x)\xrightarrow{x\rightarrow\pm\infty}(1,0,1,1),
\end{equation}
where $v=\frac{1}{\rho}$, and without loss of generality,  we still use the notation of original coordinate system $(x,t)$,
and we assume that $2\nu+\lambda=R=c_v=\kappa_0=1$. 
 
 Before giving the main theorem, we will briefly review the relevant work of \eqref{NSFAC-Lagrange} that has been done.
   Kawohl \cite{K1985}, Jiang \cite{J1994-1,J1994-2}  and Wang \cite{W2003} established the
global existence of smooth solutions for the compressible non-isentropic Navier-Stokes equations in 1-d bounded domain under the assumption $\kappa_0(1+\theta^q)\leq \kappa(\theta)\leq \kappa_1(1+\theta^q)$,  the methods they used  rely heavily on the non-degeneracy of the heat conductivity $\kappa(\theta)$. Under the degenerate and nonlinear case \eqref{kappa}, Jenssen-Karper \cite{JK2010}  proved the global existence of  weak solutions for non-slip and heat insulated boundary conditions when $\beta\in(0,3/2)$, and Pan-Zhang \cite{PZ2015} obtain the global strong solutions for more general  $\beta\in(0,\infty)$. Further,  Duan-Guo-Zhu \cite{DGZ2017} generalized Pan-Zhang's work to  stress-free and heat insulated boundary condition.  It should be pointed out that, the results of \cite{PZ2015} and \cite{DGZ2017} all require that the initial conditions  satisfies $( v_0,u_0,\theta_0)\in H^1\times H^2\times H^2$. Recently,  Huang-Shi-Sun \cite{HSS2020}, Huang-Shi \cite{HS2020} proves that the same results of \cite{PZ2015} and \cite{DGZ2017}  can be obtained for initial conditions as long as $( v_0,u_0,\theta_0)\in H^1$ is satisfied, and the large time behavior of the global strong solutions is also obtained.

For the non-isentropic compressible Navier-Stokes/Allen-Cahn system \eqref{original NSAC},  Kotschote \cite{K2012} obtained the existence and uniqueness of local strong solutions in bounded domain recently. 
 For about the isentropic Navier-Stokes/Allen-Cahn problem, Feireisl-Petzeltov$\mathrm{\acute{a}}$-Rocca-Schimperna \cite{FPRS2010} presented the global existence of a weak solution for the adiabatic exponent of pressure $\gamma>6$, and their result was improved to $\gamma>2$ by Chen-Wen-Zhu \cite{CWZ-19}.  Ding-Li-Lou \cite{DLL2013} established the existence and uniqueness of global strong solution  for initial boundary problem in 1-D.  Chen-Guo \cite{CG2017} generalized the result of \cite{DLL2013}  if the initial value contains  vacuum. Chen-He-Huang-Shi \cite{CHHS2021} discussed the global existence and uniqueness of strong solutions for this system in bounded domain with large perturbation of the initial conditions. Moreover, in order to solve the problem of moving contact lines on the solid boundary,  the generalized Navier boundary condition  and the relaxation boundary condition are established by Chen-He-Huang-Shi \cite{CHHS2022}, and the  existence and uniqueness for local strong solution in three dimensional bounded domain is obtained.

 
 \noindent\textbf{\normalsize Notations.} Throughout our paper, we denote $L^2(\mathbb{R})$ ($L^2$ without any ambiguity)  as the space of square integrable real valued function defined on $\mathbb{R}$, with the norm $\|f\|=(\int_{\mathbb{R}}|f|^2dx)^{\frac{1}{2}}$,
 and $H^l(\mathbb{R})$ ($l>0, H^l$ without any ambiguity)  the Sobolev space of $L^2$-functions $f$ on $\mathbb{R}$ whose derivatives $\partial^j_x f,  j=1,\cdots,l$ are also square integrable functions, with the norm $ \|f\|_{l}=(\sum_{j=0}^l\|{\partial^j_x f}\|^2)^{\frac{1}{2}}$. 
 
 The purpose of this paper is to gain an in-depth understanding of phase transition and phase separation phenomena for immiscible two-phase flow. We will study the existence and uniqueness of  global strong solutions for Cauchy problem of \eqref{NSFAC-Lagrange}-\eqref{initial condition} for $\beta > 0$.
The following Theorem is our main result:

\begin{theorem}\label{thm-global}
 Assume that \eqref{kappa}, and
\begin{equation}\label{condition 1}
  (v_0-1,u_0,\theta_0-1)\in H^1(\mathbb{R}),\ \ \phi_0^2-1\in L^2(\mathbb{R}),\ \ \phi_{0x}\in H^1(\mathbb{R}),
\end{equation}
and
\begin{equation}\label{condition 2}
  \inf_{x\in\mathbb{R}}v_0(x)>0,\ \ \inf_{x\in\mathbb{R}}\theta_0(x)>0,\ \ \phi_0(x)\in[-1,1].
\end{equation}
Then,  the Cauchy problem \eqref{NSFAC-Lagrange}-\eqref{initial condition} has a unique strong solution $(v,u,\theta,\phi)$ such that for fixed $T>0$, satisfying
\begin{equation}
\left\{\begin{array}{llll}
 \displaystyle v-1,u,\theta-1\in L^\infty(0,T;H^1(\mathbb{R})),\phi^2-1\in L^\infty(0,T;L^2(\mathbb{R})),\\
 \displaystyle\phi_x\in L^\infty(0,T;H^1(\mathbb{R})),\quad(\phi^2-1)_x\in L^2(0,T;L^2(\mathbb{R})),\\
\displaystyle v_x\in L^2(0,T;L^2(\mathbb{R})),u_{x},\theta_{x}\in L^2(0,T;H^1(\mathbb{R})),\phi_{xx}\in L^2(0,T;H^1(\mathbb{R})),
\end{array}\right.
\end{equation}
Moreover, there exists a positive constant $C$ depending on the initial data and $T$, satisfying
\begin{equation}\label{upper and lower bound}
C^{-1}\leq v(x,t)\leq C,\ \ C^{-1}\leq \theta(x,t)\leq C,\ \ \phi(x,t)\in[-1,1],\quad (x,t)\in\mathbb{R}\times[0,T].
\end{equation}
\end{theorem}

\begin{remark} The system \eqref{original NSAC} is proposed for the analysis and numerical simulation of immiscible two-phase flow by Heida-M$\mathrm{\acute{a}}$lek-Rajagopal  \cite{HMR2012}, Blesgen \cite{B1999}, etc. An important feature of this model is that the sharp interface between the two phases is replaced by a diffusion interface. \end{remark}

\begin{remark}The assumption \eqref{kappa} is  based on the following reason: the heat conductivity $\kappa(\theta)$  of compressible  immiscible two-phase flow vary with temperature  under very high temperature and density environment. Strictly speaking, the Chapman-Enskog expansion for the first order approximation tells us, the coefficient of heat conduction depends on temperature (see Chapman-Colwing \cite{CC1994}). \end{remark}

 \begin{remark} What we want to point out here is that, $\eqref{initial condition}$ implies that initial  concentration difference $\phi_0$ of  the immiscible two-phase flow is $1$ or $-1$ at far fields. This initial condition can be used to explain phase separation and phase transition.
  \end{remark}

  \begin{remark} Theorem 1.1   shows that no discontinuity of the phase field, vacuum,  shock
wave, mass or heat concentration will be developed in finite time as the initial data $(v_0,u_0$, $\theta_0,|\phi_0|)(x)$ $\xrightarrow{x\rightarrow\pm\infty}(1,0,1,1)$. Which means that no matter how complex the interaction between the hydrodynamic and phase-field effects, as well as the  motion of the compressible two-phase immiscible flow has large oscillations, there is no separation of the phase field in the finite time.
  \end{remark}

Now we  briefly describe some key points of proof for Theorem \ref{thm-global}.  The most important of the proof is to get the  positive upper bound and the lower bound of $v,\theta$ and $\phi$. Otherwise the system \eqref{NSFAC-Lagrange} will degenerate. 
For this purpose, firstly, inspired by the idea of Kazhikhov \cite{K1981} and Jiang \cite{J1999},  we obtain a key expression of $v$ (see \eqref{expression of v}). Secondly, using the expression of $v$, combining with the basic energy estimate \eqref{Fundamental energy inequality}, the truncation function method(see \eqref{the truncation function method}), and the convexity of $y-\ln y-1$, we get the lower bound of $v$ and $\theta$ (see \eqref{bound of density}).
Further, after observing the key inequality of $\sup_{x\in\mathbb{R}}\big(\frac{\phi_x}{v}\big)^2(x,t)$ (see \eqref{The square modulus of the first derivative for phi})
 the  upper bound of $v$ can be derived. Finally, with the help of the key inequality of $\int_0^T\sup_{x\in\mathbb{R}}(\theta-1)^{2}dt$ (see \eqref{max theta^2}),
the higher order energy estimates for $\phi$ and $v$ can be achieved through the tedious energy estimates, (see \eqref{bound of the square modulus of the first derivative}, \eqref{bound of the square modulus of the third derivative for phi}). In particular,  the upper bound of $\|\theta_x\|_{L^\infty(0,T;L^2)}$ is derived (see \eqref{energy estimate of tempeture}), and thus the upper bound of temperature $\theta$ is achieved. The whole procedure of the proof will be carried out in the next section.

\section{The Proof of Theorem}
 The local existence and uniqueness for strong solutions of \eqref{NSFAC-Lagrange}-\eqref{initial condition}  is presented as following which can be proved  by the fixed point method, the  details are omitted here.
\begin{lemma}\label{Local existence}
Let \eqref{condition 1} and \eqref{condition 2}  hold, then there exists some $T_*>0$ such that, the Cauchy problem \eqref{NSFAC-Lagrange}-\eqref{initial condition} has a unique strong solution $(v,u,\theta,\phi)$  satisfying
\begin{equation}\label{local solution space}
\left\{\begin{array}{llll}
 \displaystyle v-1,u,\theta-1\in L^\infty(0,T^*;H^1(\mathbb{R})),\phi^2-1\in L^\infty(0,T^*;L^2(\mathbb{R})),\\
 \displaystyle\phi_x\in L^\infty(0,T^*;H^1(\mathbb{R})),\quad(\phi^2-1)_x\in L^2(0,T^*;L^2(\mathbb{R})),\\
\displaystyle v_x\in L^2(0,T^*;L^2(\mathbb{R})),u_{x},\theta_{x}\in L^2(0,T^*;H^1(\mathbb{R})),\phi_{xx}\in L^2(0,T^*;H^1(\mathbb{R})),
\end{array}\right.\end{equation}
\end{lemma}

With the existence of a local solution, Theorem 1.1 can be achieved by extending the local solutions globally in time from the following series of prior estimates.
Without loss of generality, in the following prior estimates, we assume that $\nu=R=c_v=\tilde{\kappa}=1$,    and
\begin{equation}\label{initial energy}
  \int_{-\infty}^{+\infty}(v_0-1)dx=1,\ \ \int_{-\infty}^{+\infty}\big(\frac {u_0^2}{2}+(\theta_0-1)+\frac{1}{4\epsilon}(\phi_0^2-1)^2+\frac{\epsilon}{2}\frac{\phi_{0x}^2}{v_0}\big)dx=1.
\end{equation}
From here to the end of this paper,  $C>0 $ denotes  the
generic positive constant  depending only on $\|v_0-1,u_0,\theta_0-1\|_{H^1(\mathbb{R})}$, $\|\phi_0^2-1\|_{L^2(\mathbb{R})}$, $\|\phi_{0x}\|_{H^1(\mathbb{R})}$, $\inf\limits_{x\in \mathbb{R}}v_0(x)$, and $ \inf\limits_{x\in \mathbb{R}}\theta_0(x)$.

\begin{lemma}\label{Fundamental energy inequality}
Let $(v,u,\theta,\phi)$ be a smooth solution of \eqref{NSFAC-Lagrange}-\eqref{initial condition} on $(-\infty,+\infty)\times [0,T]$. Then it holds
\begin{eqnarray}\label{basic energy inequality}
 && \sup_{0\leq t\leq T}\int_{-\infty}^{+\infty}\big(\frac{u^2}{2}+\frac{1}{4\epsilon}(\phi^2-1)^2+\frac{\epsilon}{2}\frac{\phi_x^2}{v}+(v-\ln v-1)+(\theta-\ln \theta-1)\big)dx\notag\\
 &&\qquad +\int_0^TV(t)dxdt\leq E_{0},
\end{eqnarray}
where
\begin{equation}\label{E0}
 E_0\overset{\mathrm{def}}{=}\int_{-\infty}^{+\infty}\big(\frac{u_0^2}{2}+\frac{1}{4\epsilon}(\phi_0^2-1)^2+\frac{\epsilon}{2}\frac{\phi_{0x}^2}{v}+(v_0-\ln v_0-1)+(\theta_0-\ln \theta_0-1)\big)dx,
\end{equation}
and
\begin{equation}\label{V(t)}
 V(t)=\int_{-\infty}^{+\infty}\big(\frac{\theta^\beta\theta_x^2}{v\theta^2}+\frac{u_x^2}{v\theta}+\frac{v\mu^2}{\theta}\big)dx.
\end{equation}
\end{lemma}
\begin{proof}  From \eqref{NSFAC-Lagrange},  \eqref{initial condition} and  \eqref{initial energy}, we have
\begin{equation}\label{conservation}
  \int_{-\infty}^{+\infty}(v-1)dx=1,\ \ \int_{-\infty}^{+\infty}\big(\frac {u^2}{2}+(\theta-1)+\frac{1}{4\epsilon}(\phi^2-1)^2+\frac{\epsilon}{2}\frac{\phi_x^2}{v}\big)dx=1.
\end{equation}Multiplying \eqref{NSFAC-Lagrange}$_1$ by $1-\frac1v$, \eqref{NSFAC-Lagrange}$_2$ by $u$, \eqref{NSFAC-Lagrange}$_3$ by $\mu$, \eqref{NSFAC-Lagrange}$_5$ by $1-\frac1\theta$,  adding them together, we get
\begin{eqnarray}\label{basic inequality}
&& \big(\frac{u^2}{2}+\frac{1}{4\epsilon}(\phi^2-1)^2+\frac{\epsilon}{2}\frac{\phi_x^2}{v}+(v-\ln v-1)+(\theta-\ln \theta-1)\big)_t+\big(\frac{\theta^\beta\theta_x^2}{v\theta^2}+\frac{u_x^2}{v\theta}+\frac{v\mu^2}{\theta}\big) \notag\\
&&=u_x+\big(\frac{uu_x}{v}-\frac{u\theta}{v}\big)_x+\big((1-\theta^{-1})\frac{\theta^\beta\theta_x}{v}\big)_x+\epsilon\big(\frac{\phi_x\phi_t}{v}\big)_x-\frac{\epsilon}{2}\big(\frac{\phi_x^2u}{v^2}\big)_x.
\end{eqnarray}
Integrating \eqref{basic inequality} over $(-\infty,+\infty)\times[0,T]$ by parts,  \eqref{basic energy inequality} is obtained,  the proof of Lemma \ref{Fundamental energy inequality} is finished.
\end{proof}

\begin{lemma}\label{Division of area}
Let $(v,u,\theta,\phi)$ be a smooth solution of \eqref{NSFAC-Lagrange}-\eqref{initial condition} on $(-\infty,+\infty)\times [0,T]$, then  $\forall n=0,\pm1,\pm2,\cdots$,
there are points $a_{n}(t)$, $b_n(t)$ on the interval $[n,n+1]$, such that
\begin{equation}\label{bar v theta}
\begin{array}{llll}
 \displaystyle v(a_n(t),t)\overset{\mathrm{def}}{=}\bar v_n(t)= \int_n^{n+1}v(x,t)dx\in[ \alpha_1,\alpha_2],\\
  \displaystyle \theta(b_n(t),t)\overset{\mathrm{def}}{=}\bar \theta_n(t)=\int_n^{n+1}\theta(x,t)dx\in[ \alpha_1,\alpha_2],
\end{array} \end{equation}
 where $0<\alpha_1<\alpha_2$  are  the two roots of the following algebraic equation
 \begin{equation}
   y-\ln y -1=E_0.
 \end{equation}
\end{lemma}
\begin{proof}
By using the convexity of the function $y-\ln y-1$ and the Jensen's  inequality, then
\begin{equation}\label{Jensen}
\begin{array}{llll}
 \displaystyle \int_{n}^{n+1}\theta dx-\ln\int_{n}^{n+1} \theta dx  -1 \le \int_{n}^{n+1} (\theta-\ln\theta-1 ) dx,\\
\displaystyle \int_{n}^{n+1}v dx-\ln\int_{n}^{n+1} v dx-1  \le \int_{n}^{n+1} (v-\ln v-1 ) dx.
\end{array}
\end{equation}
Combining with the inequality \eqref{basic energy inequality}, using the convexity of the function $y-\ln y-1$ once again, \eqref{bar v theta} is obtained immediately.  The proof of Lemma \ref{Division of area} is finished.
\end{proof}

\begin{lemma}\label{lem-expression of v}
Let $(v,u,\theta,\phi)$ be a smooth solution of \eqref{NSFAC-Lagrange}-\eqref{initial condition} on $(-\infty,+\infty)\times [0,T]$, then $\forall n=0,\pm1,\pm2,\cdots$, it has the following expression of $v$
\begin{equation}\label{expression of v}
  v(x,t)=   D(x,t) Y(t) + \int_0^t\frac{D(x,t) Y(t)\big(\theta(x,\tau)+\frac{\epsilon}{2}\frac{\phi^2_x(x,\tau)}{v(x,\tau)}\big)}{ D(x,\tau) Y(\tau)}d\tau,\quad x\in[n,n+1],
\end{equation}
 where
 \begin{equation}\label{D}
    D(x,t)=v_0(x)e^{\int_{a_n}^{x}(u(y,t)-u_0(y))dy},
 \end{equation}
 and
 \begin{equation}\label{Y}
  Y(t)=\frac{v(a_n(t),t)}{v_0(a_n(t))}e^{-\int_0^t(\frac{\theta}{v}+\frac{\epsilon}{2}\frac{\phi^2_x}{v^{2}})(a_n(s),s)ds}.
 \end{equation}
\end{lemma}
\begin{proof}
We rewrite \eqref{NSFAC-Lagrange}$_2$ as
\begin{align}
\displaystyle (\ln v)_{xt} =\big(\frac{\theta}{v}+\frac{\epsilon}{2}\frac{\phi_x^2}{v^2}\big)_x+u_t,
\end{align}
where we have used \eqref{NSFAC-Lagrange}$_1$. Integrating the above equation over $(0,t)$, we obtain
\begin{align}\label{lnv1}
\displaystyle (\ln v)_{x} =\big(\int_0^t(\frac{\theta}{v}+\frac{\epsilon}{2}\frac{\phi_x^2}{v^2})d\tau\big)_x+u-u_0+(\ln v_0)_{x},
\end{align}
For $x\in[n,n+1]$, integrating \eqref{lnv1} from $a_n(t)$ to $x$ by parts, we have
\begin{equation}\label{v} v(x,t) = D(x,t) Y(t) e^{\int_0^t\big(\frac{\theta}{v}+\frac{\epsilon}{2}\frac{\phi^2_x}{v^2}\big)(x,s)ds},\end{equation}
with $D(x,t)$ and $Y(t)$ as defined in  \eqref{D} and \eqref{Y}  respectively. Now we introduce the function  $g(x,t)$ as following
\begin{equation}\label{g}  g(x,t) =\int_0^t\big(\frac{\theta}{v}+\frac{\epsilon}{2}\frac{\phi^2_x}{v^2}\big)(x,s)ds,
\end{equation}by using \eqref{v}, we  get the following ordinary differential equation for $g(x,t)$
\begin{equation*}  g_t=\frac{\theta(x,t)+\frac{\epsilon}{2}\frac{\phi^2_x(x,t)}{v(x,t)}}{v(x,t)}
=\frac{\theta(x,t)+\frac{\epsilon}{2}\frac{\phi^2_x(x,t)}{v(x,t)}}{ D(x,t) Y(t)e^g},\end{equation*}and this gives
\begin{equation*}
  e^g  =1+\int_0^t\frac{\theta(x,\tau)+\frac{\epsilon}{2}\frac{\phi^2_x(x,\tau)}{v(x,\tau)}}{ D(x,\tau) Y(\tau)}d\tau,\end{equation*}
substituting  the expression above into \eqref{v},  we have \eqref{expression of v}.    Thus the proof of Lemma \ref{lem-expression of v} is finished.
\end{proof}

\begin{lemma}\label{the lower bounds of density and temperature}
Let $(v,u,\theta,\phi)$ be a smooth solution of \eqref{NSFAC-Lagrange}-\eqref{initial condition} on $(-\infty,+\infty)\times [0,T]$, then it holds that for
 \begin{eqnarray}\label{bound of density}
 &&v(x,t)\geq C^{-1},\quad \theta(x,t)\geq C^{-1},\qquad\forall(x,t)\in(-\infty,+\infty)\times [0,T],\notag\\
 && \int_0^T\sup_{x\in\mathbb{R}}\theta dt+\int_0^T\int_{(\theta>2)(t)}\frac{\theta^\beta\theta_x^2}{v\theta^{2-2\eta}}dxdt\leq C,\quad \forall \frac14<\eta<\frac12.
 \end{eqnarray}
\end{lemma}
\begin{proof}
Firstly, from \eqref{condition 1}, \eqref{conservation} and the definition \eqref{D} of $D$, we have
\begin{equation}\label{upper and lower bound for D}
 C^{-1}\leq D(x,t)\leq C,\qquad  \forall(x,t)\in(-\infty,+\infty)\times [0,T].
\end{equation}
Moreover,  by using \eqref{expression of v}, we have
\begin{equation}
    Y^{-1}(t)\int_n^{n+1}v(x,t)dx= \int_n^{n+1}  D(x,t) dx+ \int_0^t\int_n^{n+1} \frac{D(x,t) \big(\theta(x,\tau)+\frac{\epsilon}{2}\frac{\phi^2_x(x,\tau)}{v(x,\tau)}\big)}{ D(x,\tau) Y(\tau)}dxd\tau,\notag
\end{equation}
Applying the inequality \eqref{bar v theta} and  \eqref{upper and lower bound for D} to the result above,  there exists a positive constant $C$, satisfying
\begin{equation}\label{Y1}
 C^{-1}Y^{-1}(t)\leq 1+\int_0^tY^{-1}(s)ds\leq C Y^{-1}(t),
 \end{equation}
which implies that
\begin{equation}\label{upper and lower bound for Y}
  0<C^{-1}\leq Y(t)\leq C<+\infty,\qquad \forall(x,t)\in(-\infty,+\infty)\times [0,T].
\end{equation}
From \eqref{v}, \eqref{upper and lower bound for D} and \eqref{upper and lower bound for Y}, we obtain the lower bound of $v$ as following
  \begin{equation}\label{the lower bound of v}
   v(x,t)\geq C^{-1},\qquad \forall(x,t)\in(-\infty,+\infty)\times[0,T].
  \end{equation}

Secondly,  Denoting by
\begin{equation}\label{set of theta}
(\theta>2)(t)\overset{\mathrm{def}}{=}\big\{x\in\mathbb{R}\big|\theta(t)>2\big\},\quad (\theta<\frac12)(t)\overset{\mathrm{def}}{=}\big\{x\in\mathbb{R}\big|\theta(t)<\frac12\big\}.
\end{equation}
By using \eqref{basic energy inequality}, one has
\begin{eqnarray*}
  E_0&\geq&\int_{(\theta<\frac12)(t)}(\theta-\ln\theta-1)dx+\int_{(\theta>2)(t)}(\theta-\ln\theta-1)dx \\
  &\geq&\big(\ln2-\frac12\big)\big|(\theta<\frac12)\big| +\big(1-\ln2\big)\big|(\theta>2)\big|, 
\end{eqnarray*}
which leading to the following inequality
\begin{equation}\label{upper bound of the set theta}
  \big|(\theta<\frac12)\big| +\big|(\theta>2)\big|\leq C.
\end{equation}
Denoting
\begin{equation}\label{the truncation function method}
  \big(\theta^{-1}-2\big)_+\overset{\mathrm{def}}{=}\max\big\{\theta^{-1}-2,0\big\},
\end{equation}
for $\forall p>2$, multiplying \eqref{NSFAC-Lagrange}$_3$ by $\theta^{-2}\big(\theta^{-1}-2\big)_+^p$, integrating over $(-\infty,+\infty)$  with respect to $x$, combining with \eqref{the lower bound of v} and \eqref{upper bound of the set theta},  one has
\begin{align}\label{lower bound of theta}
&\frac{1}{p+1}\frac{d}{dt}\int_{-\infty}^{+\infty}\big( {\theta}^{-1}-2\big)_+^{p+1} dx+\int_{-\infty}^{+\infty}\big(\frac{u_x^2}{v\theta^2}+\frac{v\mu^2}{\theta^2}\big) \big({\theta}^{-1}-2\big)_+^{p}dx\notag\\
&\qquad+2\int_{-\infty}^{+\infty}\frac{\theta_x^2}{v\theta^3}\big( {\theta}^{-1}-2\big)_+^{p}dx+p\int_{-\infty}^{+\infty}\frac{\theta_x^2}{v\theta^2}\big( {\theta}^{-1}-2\big)_+^{p-1}dx\notag\\
 &\leq\int_{-\infty}^{+\infty}\frac{u_x}{v\theta}\big( {\theta}^{-1}-2\big)_+^{p} dx\\ 
 & \leq\frac{1}{2}\int_{-\infty}^{+\infty}\frac{u_x^2}{v\theta^2}\big( {\theta}^{-1}-2\big)_+^{p}dx+C\int_{-\infty}^{+\infty}\frac{1}{v}\big( {\theta}^{-1}-2\big)_+^{p}dx\notag\\ 
 &\leq\frac{1}{2}\int_{-\infty}^{+\infty}\frac{u_x^2}{v\theta^2}\big( {\theta}^{-1}-2\big)_+^{p} dx+C\Big(\int_{-\infty}^{+\infty}\big( {\theta}^{-1}-2\big)_+^{p+1}\Big)^{\frac{p}{p+1}}.\notag\end{align}
Applying Gronwall's inequality to  the above result \eqref{lower bound of theta}, one derives
\begin{equation}\label{lower bound of theta Lp norm}  \sup_{0\leq t\leq T}\left\|  \big( {\theta}^{-1}-2\bar\theta\big)_+(\cdot,t) \right\|_{L^{p+1}(\mathbb{R})}\leq C,\ \ \forall p>2,
\end{equation}
where $C$ is independent of $p$, and further, taking the limit as $p\rightarrow+\infty$ on both sides of  \eqref{lower bound of theta Lp norm}, one  eventually gets 
\begin{equation}\label{lower bound of theta L-infty norm} 
 \sup_{0\leq t\leq T}\left\|  \big( {\theta}^{-1}-2\big)_+(\cdot,t) \right\|_{L^{\infty}(\mathbb{R})}\leq C,
\end{equation}
which the positive lower bound of temperature $\theta$ is thus proved.

Finally, setting
\begin{equation}\label{eta}
  \eta\overset{\mathrm{def}}{=}\frac{1}{4}\max\{1,2-\beta\}\in\big(\frac14,\frac{1}{2}\big),
  \end{equation}
 from Cauchy inequality, combining with \eqref{basic energy inequality}, one has
\begin{align}\label{theta0}
&\int_0^T\sup_{x\in\mathbb{R}}\theta dt\leq C\int_0^T\sup_{x\in\mathbb{R}}\int_{-\infty}^{x}\partial_y\big(\theta-2\big)_+dydt+C\notag \\
&\leq C\int_0^T\int_{(\theta>2)(t)}\frac{\theta^\beta \theta_x^2}{v\theta^{2-2\eta}}dxdt+C\int_0^T\int_{(\theta>2)(t)}\frac{v\theta^{2-2\eta}}{\theta^\beta} dxdt +C\\
&\leq C\int_0^T\int_{(\theta>2)(t)}\frac{\theta^\beta \theta_x^2}{v\theta^{2-2\eta}}dxdt+C\int_0^T\sup_{x\in\mathbb{R}}\theta^{\max\{2-2\eta-\beta,0\}}\int_{(\theta>2)(t)}v dxdt +C\notag\\
&\leq C\int_0^T\int_{(\theta>2)(t)}\frac{\theta^\beta \theta_x^2}{v\theta^{2-2\eta}}dxdt+\delta\int_0^T\sup_{\mathbb{R}}\theta dt+C(\delta),\notag
\end{align}
and it immediately follows that
\begin{equation}\label{The integral of theta}
  \int_0^T\sup_{x\in\mathbb{R}}\theta dt\leq \frac{C}{1-\delta}\int_0^T\int_{(\theta>2)(t)}\frac{\theta^\beta \theta_x^2}{v\theta^{2-2\eta}}dxdt+C(\delta).
\end{equation}
Multiplying \eqref{NSFAC-Lagrange}$_3$ by $\big(\theta^{\eta}-(2)^{\eta}\big)_+\theta^{\eta-1}$ and integrating over $(-\infty,+\infty)\times[0,T]$ by parts, combining  \eqref{upper bound of the set theta} and \eqref{basic energy inequality}, then 
\begin{align*}
 & (1-2\eta)\int_0^T\int_{(\theta>2)(t)}\frac{\theta^\beta\theta_x^2}{v\theta^{2-2\eta}}dxdt+\int_0^T\int_{-\infty}^{+\infty}\Big(\frac{u_x^2}{v}+v\mu^2\Big)
\Big(\theta^{\eta}-2^{\eta}\Big)_+\theta^{\eta-1}dx\nonumber \\
 &=\frac{1}{2\eta}\int_{-\infty}^{+\infty}\Big(\big(\theta^{\eta}-2^{\eta}\big)_+^2-\big(\theta_0^{\eta}-(2\bar\theta)^{\eta}\big)_+^2\Big)dx+2^\eta(1-\eta)
 \int_0^T\int_{(\theta>2)(t)}\frac{\theta^\beta\theta_x^2}{v\theta^{2-\eta}}dxdt\nonumber\\
 &\qquad+\int_0^T\int_{-\infty}^{+\infty}\frac{\theta u_x}{v}\Big(\theta^{\eta}-2^{\eta}\Big)_+\theta^{\eta-1}dxdt\\
&\leq C+\frac{1-2\eta}2\int_0^T\int_{(\theta>2)(t)}\frac{\theta^\beta\theta_x^2}{v\theta^{2-2\eta}}dxdt+
\frac12\int_0^T\int_{-\infty}^{+\infty}\frac{u_x^2}{v}\Big(\theta^{\eta}-2^{\eta}\Big)_+\theta^{\eta-1}dxdt  \nonumber\\
&\qquad+C\int_0^T\int_{-\infty}^{+\infty}\frac{\theta^2}{v}\Big(\theta^{\eta}-2^{\eta}\Big)_+\theta^{\eta-1}dx dt,\nonumber
\end{align*}
which yields 
\begin{equation}\label{auxiliary inequality for estimation of temperature 1}
\left.\begin{array}{llll}
\displaystyle\frac{1-2\eta}{2}\int_0^T\int_{(\theta>2)(t)}\frac{\theta^\beta\theta_x^2}{v\theta^{2-2\eta}}dxdt+
\int_0^T\int_{-\infty}^{+\infty}\Big(\frac{u_x^2}{2v}+v\mu^2\Big)
\displaystyle\Big(\theta^{\eta}-2^{\eta}\Big)_+\theta^{\eta-1}dx \\
\displaystyle\leq C\int_0^T\int_{(\theta>2)(t)}\theta^{\eta+1}\Big(\theta^{\eta}-2^{\eta}\Big)dx dt+C \\
\displaystyle\leq C\int_0^T\sup_{\mathbb{R}}\theta^{2\eta}\int_{(\theta>2)(t)}\theta dx dt+C\int_0^T\sup_{\mathbb{R}}\theta^{2\eta}\int_{(\theta>2)(t)}\theta^{1-\eta} dx dt+C\\
\displaystyle \leq C\int_0^T\sup_{x\in\mathbb{R}}\theta^{2\eta} dt+C\leq\delta \int_0^T\sup_{x\in\mathbb{R}}\theta dt+C(\delta).
\end{array}\right.
\end{equation}
Thus, combining \eqref{auxiliary inequality for estimation of temperature 1} and \eqref{The integral of theta}, one  derives
\begin{equation}\label{integral sup theta}
  \int_0^T\sup_{x\in\mathbb{R}}\theta dt+\int_0^T\int_{(\theta>2)(t)}\frac{\theta^\beta\theta_x^2}{v\theta^{2-2\eta}}dxdt\leq C.
\end{equation}
  and moreover, together with \eqref{basic energy inequality}, one obtains
\begin{equation}\label{integral sup theta 2}
  \int_0^T\int_{-\infty}^{+\infty}\frac{\theta^\beta\theta_x^2}{v\theta^{2-2\eta}}dxdt\leq C.
\end{equation}
The proof of Lemma 2.5 is completed. 
  \end{proof}

\begin{lemma}\label{the upper bounds of density 1}
Let $(v,u,\theta,\phi)$ be a smooth solution of \eqref{NSFAC-Lagrange}-\eqref{initial condition} on $(-\infty,+\infty)\times [0,T]$, then $\forall n=0,\pm1,\pm2,\cdots$,  it has the following inequalities:
\begin{equation}\label{bound of phi and v}
|\phi(x,t)|\leq C,\qquad v(x,t)\leq C,\qquad \forall (x,t)\in (-\infty,+\infty)\times[0,T],
\end{equation}
where $C$ is only dependent of $\epsilon$, $E_0$.
\end{lemma}
\begin{proof}
 Firstly, $\forall n=0,\pm1,\pm2,\cdots$,  from \eqref{basic energy inequality}, we have
  \begin{eqnarray}
   \frac{\epsilon}{4}\int_n^{n+1}\phi^4(x,t)dx&\leq& \frac{\epsilon}{2}\int_n^{n+1}\phi^2(x,t)dx+\frac{\epsilon}{4}+E_0\notag\\
   &\leq& \frac{\epsilon}{8}\int_n^{n+1}\phi^4(x,t)dx+\frac{3\epsilon}{4}+E_0,\notag
 \end{eqnarray}
 which implies that
  \begin{equation}\label{phi L4}
   \int_n^{n+1}\phi^4(x,t)dx\leq 6+\frac{8E_0}{\epsilon},
 \end{equation}
and therefore
 \begin{equation}\label{phi L1}
   \int_n^{n+1}\phi(x,t) dx\leq C.
 \end{equation}
where $C$ is independent of $n$.
Now $\forall (x,t)\in[n,n+1)\times[0,T]$,
\begin{eqnarray}\label{lower bound of phi}
  |\phi(x,t)|&\leq&\left|\int_n^{n+1}\big(\phi(x,t)-\phi(y,t)\big)dy\right|+\left|\int_n^{n+1}\phi(y,t)dy\right|\notag \\
 &\leq&\left|\int_n^{n+1}\Big(\int_y^x\phi_\xi(\xi,t)d\xi\Big)dy\right| +C\\
  &\leq& \Big(\int_n^{n+1}\frac{\phi_x^2}{v}dx\Big)^{\frac12}+C\leq\big(\frac{2}{\epsilon}E_0\big)^\frac12+ C.\notag
\end{eqnarray}

Finally, we can give the upper bounds  of $v$. In fact,  combining the  expression of $v$ \eqref{expression of v}  with the upper and lower bound estimates \eqref{upper and lower bound for D}--\eqref{the lower bound of v}, we have the following inequality
\begin{eqnarray}\label{the upper bound estimate for v 1}
  v(x,t)&=& D(x,t) Y(t) + \int_0^t\frac{D(x,t) Y(t)\big(\theta(x,\tau)+\frac{\epsilon}{2}\frac{\phi^2_x(x,\tau)}{v(x,\tau)}\big)}{ D(x,\tau) Y(\tau)}d\tau\\
&\le& C+C\int_0^t\Big(\sup_{x\in\mathbb{R}} \theta (x,\tau)+\sup_{x\in\mathbb{R}}\big(\frac{\phi_x(x,\tau)}{v(x,\tau)}\big)^2\sup_{x\in\mathbb{R}}v(x,\tau)\Big)d\tau.\notag
\end{eqnarray}
Deriving from \eqref{NSFAC-Lagrange}$_5$,  we have
\begin{equation}\label{second derivative form for phi}
 \epsilon\Big(\frac{\phi_x}{v}\Big)_x=-\mu+\frac{1}{\epsilon}(\phi^3-\phi),
\end{equation}
and then from \eqref{basic energy inequality}, \eqref{conservation},\eqref{bound of density}, \eqref{lower bound of phi}, we obtain
\begin{equation}\label{estimate for second derivative form of phi}
\int_{-\infty}^{+\infty}\Big(\frac{\phi_x}{v}\Big)_x^2\frac{v}{\theta}dx\leq C\big(1+V(t)\big).
\end{equation}
Using \eqref{basic energy inequality},\eqref{conservation},\eqref{the lower bound of v},\eqref{lower bound of phi}, \eqref{estimate for second derivative form of phi}, we get
\begin{eqnarray}\label{The square modulus of the first derivative for phi}
 \sup_{x\in\mathbb{R}}\big(\frac{\phi_x}{v}\big)^2(x,t)& \leq& C\int_{-\infty}^{+\infty}\frac{\phi_x}{v}\Big(\frac{\phi_x}{v}\Big)_xdx\notag\\
 &\leq&C\int_{-\infty}^{+\infty}\frac{\theta}{v^2}\frac{\phi_x^2}{v}dx+\int_{-\infty}^{+\infty}\Big(\frac{\phi_x}{v}\Big)_x^2\frac{v}{\theta}dx\\
  &\leq& C\big(\sup_{x\in\mathbb{R}}\theta+1+V(t)\big).\notag
\end{eqnarray}
 Substituting \eqref{The square modulus of the first derivative for phi} into \eqref{the upper bound estimate for v 1}, we achieve
 \begin{eqnarray}\label{the upper bound estimate for v}
  v(x,t)\leq C+C\int_0^t\big(\sup_{x\in\mathbb{R}} \theta+1+V(t)\big)\sup_{x\in\mathbb{R}} v(x,\tau)d\tau.
\end{eqnarray}
Applying the Gronwall inequality to the above \eqref{the upper bound estimate for v}, combining with Lemma 2.2 and Lemma 2.5, we get
\begin{equation}\label{the upper bound of v}
 v(x,t)\leq C,\qquad\forall (x,t)\in (-\infty.+\infty)\times [0,T].
\end{equation}
The proof of Lemma \ref{the upper bounds of density 1} is finished.
\end{proof}

\begin{lemma}
Let $(v,u,\theta,\phi)$ be a smooth solution of \eqref{NSFAC-Lagrange}-\eqref{initial condition} on $(-\infty,+\infty)\times [0,T]$,  then for $\forall(x,t)(-\infty,+\infty)\times [0,T]$, it holds that
 \begin{equation}\label{ux2andthetax2}
\int_0^T\int_{-\infty}^{+\infty}\big(u_x^2+\frac{\theta_x^2}\theta \big)dxdt\leq C,
 \end{equation}
 where the positive constant $C$ only depends on $T$, the initial data $v_0,u_0,\theta_0,\chi_0$ and $\epsilon,h$.
 \end{lemma}
 \begin{proof}
 Multiplying \eqref{NSFAC-Lagrange}$_2$ by $u$ and integrating the resultant with respect of $x$ over $\mathbb{R}$, by using \eqref{basic energy inequality}, \eqref{E0},\eqref{bound of density}, \eqref{The square modulus of the first derivative for phi} and \eqref{integral sup theta}, one has
 \begin{equation}\label{thetax2}
 \left.\begin{array}{llll}
 \displaystyle \frac12\frac{d}{dt}\int_{-\infty}^{+\infty}u^2dx+\int_{-\infty}^{+\infty}\frac{u_x^2}vdx\\
 \displaystyle=\int_{-\infty}^{+\infty}\frac{\theta}{v}u_xdx+\frac\epsilon2\int_{-\infty}^{+\infty}\frac{\chi_x^2}{v^2}u_xdx\\
\displaystyle \leq C\int_{-\infty}^{+\infty}\frac{|\theta-1|}{v}|u_x|dx+C\int_{-\infty}^{+\infty}\frac{|v-1|}{v}|u_x|dx +\int_{-\infty}^{+\infty}\frac{\chi_x^4}{v^3}dx+\frac14\int_{-\infty}^{+\infty}\frac{u_x^2}vdx\\ 
\displaystyle \leq C\int_{-\infty}^{+\infty}(\theta-1)^2dx+C\int_{-\infty}^{+\infty}(v-1)^2dx+\int_{-\infty}^{+\infty}\frac{\chi_x^4}{v^3}dx+\frac12\int_{-\infty}^{+\infty}\frac{u_x^2}vdx\\
\displaystyle\leq C+C\int_{(\theta>2)(t)}\theta^2dx+C\big(\sup_{x\in\mathbb{R}} \theta+1+V(t)\big)+\frac12\int_{-\infty}^{+\infty}\frac{u_x^2}vdx\\
\displaystyle\leq C+C\big(\sup_{x\in\mathbb{R}} \theta+V(t)\big)+\frac12\int_{-\infty}^{+\infty}\frac{u_x^2}vdx.
 \end{array}
 \right.
 \end{equation}
 Integrating \eqref{thetax2} over $[0,T]$, combining with \eqref{basic energy inequality} and \eqref{integral sup theta}, one gets
 
 \begin{equation}\label{ux2}
   \int_0^T\int_{-\infty}^{+\infty}u_x^2dxdt\leq C,
 \end{equation}
 further, combining with \eqref{integral sup theta 2} and \eqref{bound of density}, one obtains
  \begin{equation}\label{thetax2}
   \int_0^T\int_{-\infty}^{+\infty}\frac{\theta_x^2}{\theta}dxdt\leq 
  C\int_0^T\int_{-\infty}^{+\infty}\frac{\theta^\beta\theta_x^2}{v\theta^{2-2\eta}}\theta^{1-2\eta-\beta}dxdt\leq C.
 \end{equation}
 The proof of Lemma 2.7 is completed.
 \end{proof}

\begin{lemma}\label{the square modulus estimate for density}
Let $(v,u,\theta,\phi)$ be a smooth solution of \eqref{NSFAC-Lagrange}-\eqref{initial condition} on $(-\infty,+\infty)\times [0,T]$,  then for $\forall(x,t)(-\infty,+\infty)\times [0,T]$, the following inequalities hold
 \begin{equation}\label{bound of the square modulus of the first derivative}
\sup_{0\leq t\leq T}\int_{-\infty}^{+\infty}v_x^2dx\leq C,\qquad \int_0^T\int_{-\infty}^{+\infty}\Big((\phi^2-1)_x^2+\phi_{xx}^2+\phi_t^2\Big)dxdt\leq C.
 \end{equation}
\end{lemma}
\begin{proof} Firstly, we rewrite  \eqref{NSFAC-Lagrange}$_2$ as following
\begin{equation}\label{Another form of the momentum equation}
 \Big(u-\frac{v_x}{v}\Big)_t=-\Big(\frac{\theta}{v}+\frac\epsilon2\big(\frac{\phi_x}{v}\big)^2\Big)_x
\end{equation}
Multiplying \eqref{Another form of the momentum equation} by $u-\frac{v_x}{v}$, integrating by parts over $(-\infty,+\infty)\times [0,T]$, we have
\begin{equation}\label{the square modulus estimate for density-1}
\left.\begin{array}{llll}
\displaystyle\frac12\int_{-\infty}^{+\infty}\Big(u-\frac{v_x}{v}\Big)^2(x,t)dx-\frac12\int_{-\infty}^{+\infty}\Big(u_0-\frac{v_x}{v}(x,0)\Big)^2dx \notag\\
\displaystyle=\int_0^T\int_{-\infty}^{+\infty}\Big(\frac{\theta v_x}{v^2}-\frac{\theta_x}{v}-\epsilon\frac{\phi_x}{v}\big(\frac{\phi_x}{v}\big)_x\Big)\Big(u-\frac{v_x}{v}\Big)dxdt\notag \\
\displaystyle=-\int_0^T\int_{-\infty}^{+\infty}\frac{\theta v_x^2}{v^3}dxdt+\int_0^T\int_{-\infty}^{+\infty}\frac{\theta u v_x}{v^2}dxdt\\
\displaystyle\ \ \ \ \ -\int_0^T\int_{-\infty}^{+\infty}\frac{\theta_x}{v}\Big(u-\frac{v_x}{v}\Big)dxdt
-\int_0^T\int_{-\infty}^{+\infty}\epsilon\frac{\phi_x}{v}\big(\frac{\phi_x}{v}\big)_x\Big(u-\frac{v_x}{v}\Big)dxdt.\notag
  \end{array}
  \right.
  \end{equation}
Now we give the last three terms on the right side of \eqref{the square modulus estimate for density-1}. First, by using \eqref{basic energy inequality}, \eqref{the lower bound of v}, \eqref{bound of density} and \eqref{the upper bound of v},  we have
\begin{equation}\label{I1}
\left.\begin{array}{llll}
\displaystyle \Big|\int_0^T\int_{-\infty}^{+\infty}\frac{\theta u v_x}{v^2}dxdt\Big| &\leq&\displaystyle\frac{1}{2}\int_0^T\int_{-\infty}^{+\infty}\frac{\theta v_x^2}{v^3}dxdt+\frac{1}{2}\int_0^T\int_{-\infty}^{+\infty}\frac{u^2\theta}{v}dxdt\\
\displaystyle &\leq&\displaystyle\frac{1}{2}\int_0^T\int_{-\infty}^{+\infty}\frac{\theta v_x^2}{v^3}dxdt+C\int_0^T\sup_{x\in\mathbb{R}}\theta dt\\
 \displaystyle&\leq&\displaystyle\frac{1}{2}\int_0^T\int_{-\infty}^{+\infty}\frac{\theta v_x^2}{v^3}dxdt+C.
 \end{array}\right.
\end{equation}
Next, by using \eqref{basic energy inequality} and \eqref{bound of density} we obtain
\begin{equation}\label{I2}
\left.\begin{array}{llll}
\displaystyle\Big| \int_0^T\int_{-\infty}^{+\infty}\frac{\theta_x}{v}\Big(u-\frac{v_x}{v}\Big)dxdt\Big|\\
\displaystyle\leq C\int_0^T\int_{-\infty}^{+\infty}\frac{\theta_x^2}{\theta}dxd\tau+C \int_0^T\int_{-\infty}^{+\infty}\frac{\theta}{v^2}\Big(u-\frac{v_x}{v}\Big)^2dxd\tau\\
\displaystyle\leq C+C\int_0^T\sup_{x\in\mathbb{R}}\theta\int_{-\infty}^{+\infty}\Big(u-\frac{v_x}{v}\Big)^2dxd\tau.
\end{array}
\right.
\end{equation}
Furthermore, from \eqref{bound of density}, \eqref{estimate for second derivative form of phi}, \eqref{the upper bound of v}, we have
\begin{eqnarray}\label{I3}
 &&\Big|\int_0^T\int_{-\infty}^{+\infty}\epsilon\frac{\phi_x}{v}\big(\frac{\phi_x}{v}\big)_x\Big(u-\frac{v_x}{v}\Big)dxdt\Big|\notag\\
 &&\leq C\int_0^T\int_{-\infty}^{+\infty}\Big(\big|(\frac{\phi_x}{v})_x\big|^2+\big|\frac{\phi_x}{v}\big|^2\big(u-\frac{v_x}{v}\big)^2\Big)dxdt\notag\\
 && \leq C\int_0^T\int_{-\infty}^{+\infty}\phi^2_{xx}dxdt+C\int_0^T\big(\sup_{x\in\mathbb{R}}\theta+1+V(t)\big)\int_{-\infty}^{+\infty}v_{x}^2dxdxdt\\
 &&\qquad+
 C\int_0^T\sup_{x\in\mathbb{R}}\big|\frac{\phi_x}{v}\big|^2\int_{-\infty}^{+\infty}\big(u-\frac{v_x}{v}\big)^2dxdt.\notag 
\end{eqnarray}
Secondly, rewriting  \eqref{NSFAC-Lagrange}$_{3,4}$ as follows
\begin{equation}\label{Allen-Cahn-Lagrange}
 \phi_t-\epsilon\phi_{xx}=-\epsilon\frac{\phi_xv_x}{v}-\frac{v}{\epsilon}\big(\phi^3-\phi\big).
\end{equation}
Multiplying \eqref{Allen-Cahn-Lagrange} by $\phi_{xx}$, integrating the resultant over $(-\infty,+\infty)$, with respect to $x$,  combining with \eqref{bound of density}, \eqref{bound of phi and v}, \eqref{The square modulus of the first derivative for phi} and \eqref{lower bound of phi},  we obtain
\begin{align}
&\quad\frac{1}{2}\frac{d}{dt}\int_{-\infty}^{+\infty}\phi_x^2dx+\epsilon\int_{-\infty}^{+\infty}\phi_{xx}^2dx+\frac{1}{2\epsilon}\int_{-\infty}^{+\infty}(\phi^2-1)_x^2dx\notag\\
&=\epsilon\int_{-\infty}^{+\infty}\frac{\phi_xv_x}{v}\phi_{xx}dx+ \frac{1}{\epsilon}\int_{-\infty}^{+\infty}(1-\phi^2)\phi_x^2dx\notag\\
& \leq C\Big(\int_{-\infty}^{+\infty}\phi_x^2v_x^2dx+
\int_{-\infty}^{+\infty}\phi_x^2dx\Big)+\frac{\epsilon}{2}\int_{-\infty}^{+\infty}\phi_{xx}^2dx\\
&\leq  C\Big(\sup_{x\in\mathbb{R}}\phi_x^2(x,t)\int_{-\infty}^{+\infty}v_{x}^2dx
+
\int_{-\infty}^{+\infty}\phi_x^2dx\Big)+\frac{\epsilon}{2}\int_{-\infty}^{+\infty}\phi_{xx}^2dx\notag\\
&\leq   C\Big(\big(\sup_{x\in\mathbb{R}}\theta+1+V(t)\big)\int_{-\infty}^{+\infty}v_{x}^2dx+
\int_{-\infty}^{+\infty}\phi_x^2dx\Big)+\frac{\epsilon}{2}\int_{-\infty}^{+\infty}\phi_{xx}^2dx,\notag
\end{align}
combining with \eqref{basic energy inequality}, \eqref{bound of density}, \eqref{bound of phi and v}, \eqref{I1}, \eqref{I2}, \eqref{I3}, \eqref{the square modulus estimate for density-1}, from \eqref{The square modulus of the first derivative for phi}, \eqref{max theta^2} and Gronwall's inequality, we obtain
\begin{equation}\label{estimate for v-x^2}
\sup_{0\leq t\leq T} \int_{-\infty}^{+\infty}(v_x^2+\phi_x^2)dx+\int_0^T\int_{-\infty}^{+\infty}\Big((\phi^2-1)_x^2+\phi_{xx}^2+\frac{\theta v_x^2}{v^3}\Big)dxdt\leq C.
\end{equation}

Finally,  from \eqref{Allen-Cahn-Lagrange}, we have
\begin{equation}\label{Allen-Cahn-Lagrange-0}
 \phi_t=\epsilon\phi_{xx}-\epsilon\frac{\phi_xv_x}{v}-\frac{v}{\epsilon}\big(\phi^3-\phi\big),
\end{equation}
integrating \eqref{Allen-Cahn-Lagrange-0} over $(-\infty,+\infty)$, we obtain
\begin{eqnarray}\label{estimate of phi-xx-l2}
 \int_{-\infty}^{+\infty}\phi_t^2dx&\leq&C\Big(\int_{-\infty}^{+\infty}\phi_{xx}^2dx+\int_{-\infty}^{+\infty}\phi_x^2v_x^2dx+\int_{-\infty}^{+\infty}\big(\phi^3-\phi\big)^2dx\Big)\notag\\
 &\leq&C\Big(\int_{-\infty}^{+\infty}\phi_{xx}^2dx+\int_{-\infty}^{+\infty}v_x^2dx\int_{-\infty}^{+\infty}\phi_{xx}^2dx+1\Big)\\
 &\leq&C\Big(\int_{-\infty}^{+\infty}\phi_{xx}^2dx+1\Big),\notag
\end{eqnarray}
then, from \eqref{estimate for v-x^2}, we achieve
\begin{equation}\label{phi-t-L2}
  \int_0^T\int_{-\infty}^{+\infty}\phi_t^2dxdt\leq C.
\end{equation}
The proof of Lemma \ref{the square modulus estimate for density} is completed.
\end{proof}

\begin{lemma}\label{phi-xxx}
Let $(v,u,\theta,\phi)$ be a smooth solution of \eqref{NSFAC-Lagrange}-\eqref{initial condition} on $(-\infty,+\infty)\times [0,T]$, then  for $\forall(x,t)\in(-\infty,+\infty)\times [0,T]$, the following inequality holds
 \begin{equation}\label{bound of the square modulus of the third derivative for phi}
\sup_{0\leq t\leq T}\int_{-\infty}^{+\infty}\phi_{xx}^2dx+\int_0^T\int_{-\infty}^{+\infty}\big(\phi_{xt}^2+\big(\frac{\phi_{x}}{v}\big)_{xx}^2\big)dxdt\leq C.
 \end{equation}
\end{lemma}
\begin{proof}
Rewriting  \eqref{Allen-Cahn-Lagrange} as
\begin{equation}\label{Allen-Cahn-Lagrange-1}
\frac{\phi_t}{v}-\epsilon\Big(\frac{\phi_x}{v}\Big)_{x}=-\frac{1}{\epsilon}(\phi^3-\phi),
\end{equation}
differentiating \eqref{Allen-Cahn-Lagrange-1} with respect to $x$, we obtain
\begin{equation}\label{Allen-Cahn-Lagrange-2}
 \Big(\frac{\phi_x}{v}\Big)_t-\epsilon\Big(\frac{\phi_x}{v}\Big)_{xx}=-\frac{1}{\epsilon}\big(\phi^3-\phi\big)_x+\frac{\phi_tv_x}{v^2}-\frac{\phi_x u_x}{v^2},
\end{equation}
multiplying \eqref{Allen-Cahn-Lagrange-2} by $\big(\frac{\phi_{x}}{v}\big)_t$, integrating the resultant over $(-\infty,+\infty)$, from \eqref{basic energy inequality}, \eqref{bound of density}, \eqref{bound of phi and v}, \eqref{bound of the square modulus of the first derivative},  \eqref{estimate of phi-xx-l2}, we have
\begin{equation*}
\left.\begin{array}{llll}
\displaystyle\int_{-\infty}^{+\infty}\Big(\frac{\phi_{x}}{v}\Big)_t^2dx+ \frac{\epsilon}{2}\frac{d}{dt}\int_{-\infty}^{+\infty}\Big(\frac{\phi_x}{v}\Big)_{x}^2dx \\
\displaystyle =-\frac{1}{\epsilon}\int_{-\infty}^{+\infty}\big(\phi^3-\phi\big)_x\big(\frac{\phi_{x}}{v}\big)_tdx+\int_{-\infty}^{+\infty}\frac{\phi_tv_x}{v^2}
\big(\frac{\phi_{x}}{v}\big)_tdx-\int_{-\infty}^{+\infty}\frac{\phi_x u_x}{v^2}\big(\frac{\phi_{x}}{v}\big)_tdx\\
\displaystyle\leq C\Big(\int_{-\infty}^{+\infty}\big(3\phi^2-1\big)^2\phi_x^2dx+\int_{-\infty}^{+\infty}\phi_t^2v_x^2dx+\int_{-\infty}^{+\infty}\phi_x^2 u_x^2dx\Big)+\frac{1}{3}\int_{-\infty}^{+\infty}\big(\frac{\phi_{x}}{v}\big)_t^2dx \\
\displaystyle\leq C\Big(1+\|\phi_t\|_{L^\infty}^2\int_{-\infty}^{+\infty}v_x^2dx+\|\frac{\phi_{x}}{v}\|_{L^\infty}^2\int_{-\infty}^{+\infty}u_x^2dx\Big)
+\frac{1}{3}\int_{-\infty}^{+\infty}\big(\frac{\phi_{x}}{v}\big)_t^2dx\\
\displaystyle\leq C\Big(1+\int_{-\infty}^{+\infty}\big(\phi_t^2+2|\phi_t\phi_{xt}|\big)dx+\int_{-\infty}^{+\infty}\Big(\frac{\phi_x}{v}\Big)_{x}^2dx\int_{-\infty}^{+\infty}u_x^2dx\Big)
+\frac{1}{3}\int_{-\infty}^{+\infty}\big(\frac{\phi_{x}}{v}\big)_t^2dx\\
\displaystyle\leq C\Big(1+\int_{-\infty}^{+\infty}\phi_{xx}^2dx+\varepsilon\int_{-\infty}^{+\infty}\phi_{xt}^2dx+\int_{-\infty}^{+\infty}
\Big(\frac{\phi_x}{v}\Big)_{x}^2dx\int_{-\infty}^{+\infty}u_x^2dx\Big)
+\frac{1}{3}\int_{-\infty}^{+\infty}\big(\frac{\phi_{x}}{v}\big)_t^2dx\\
\displaystyle\leq C\Big(1+\int_{-\infty}^{+\infty}\phi_{xx}^2dx+\int_{-\infty}^{+\infty}u_x^2dx\int_{-\infty}^{+\infty}\Big(\frac{\phi_x}{v}\Big)_{x}^2dx\Big)
+\frac{1}{2}\int_{-\infty}^{+\infty}\big(\frac{\phi_{x}}{v}\big)_t^2dx,
\end{array}\right.
\end{equation*}
where in the last inequality  $\phi_{xt}=\big(\frac{\phi_x}{v}\big)_t v+\frac{\phi_x u_x}{v}$ is used. Therefore, from Gronwall's inequality, we get
\begin{equation}\label{higher derivative energy estimation for phi}
 \sup_{t\in[0,T]} \int_{-\infty}^{+\infty}\Big(\frac{\phi_x}{v}\Big)_{x}^2dx+\int_0^T\int_{-\infty}^{+\infty}\Big(\frac{\phi_{x}}{v}\Big)_t^2dxdt\leq C.
\end{equation}
 Combining with  \eqref{basic energy inequality}, \eqref{bound of the square modulus of the first derivative}, we  have
\begin{equation}
\sup_{0\leq t\leq T}\int_{-\infty}^{+\infty}\phi_{xx}^2dx+\int_0^T\int_{-\infty}^{+\infty}\phi_{xt}^2dxdt\leq C.
 \end{equation}
Furthermore,  using the Sobolev embedding theorem, follows from \eqref{basic energy inequality} and \eqref{higher derivative energy estimation for phi}, we obtain
\begin{equation}\label{the upper and lower bounds of the derivative for phi}
  \sup_{(x,t)\in\mathbb{R}\times[0,T]}\big|\frac{\phi_x}{v}\big|\leq \sqrt{2}\big\|\frac{\phi_x}{v}\big\|\Big\|\big(\frac{\phi_x}{v}\big)_{x}\Big\|\leq C.
\end{equation}
Moreover, from \eqref{Allen-Cahn-Lagrange-2} and the inequalities obtained above, we achieve
\begin{equation}
\int_0^T\int_{-\infty}^{+\infty}\big(\frac{\phi_{x}}{v}\big)_{xx}^2dxdt\leq C,
 \end{equation}
  the proof of Lemma \ref{phi-xxx} is finished.
\end{proof}

\begin{lemma}\label{velocity-x}
Let $(v,u,\theta,\phi)$ be a smooth solution of \eqref{NSFAC-Lagrange}-\eqref{initial condition} on $(-\infty,+\infty)\times [0,T]$, then  for $\forall(x,t)\in(-\infty,+\infty)\times [0,T]$, the following inequality holds
 \begin{equation}\label{energy estimate of velocity}
\sup_{0\leq t\leq T}\int_{-\infty}^{+\infty}u_{x}^2dx+\int_0^T\int_{-\infty}^{+\infty}\big(u_{t}^2+u^2_{xx}\big)dxdt\leq C.
 \end{equation}
\end{lemma}
\begin{proof}
Multiplying \eqref{NSFAC-Lagrange}$_2$ by $u_{xx}$ and integrating the resultant over $(-\infty,+\infty)\times(0,T)$, by using \eqref{bound of density}, \eqref{bound of phi and v}, \eqref{bound of the square modulus of the first derivative},  \eqref{higher derivative energy estimation for phi}, \eqref{the upper and lower bounds of the derivative for phi},  we obtain
\begin{equation}\label{u-x and u-xx 1}
\left.\begin{array}{llll}
\displaystyle\frac{1}{2}\int_{-\infty}^{+\infty}u_x^2dx+\int_0^T\int_{-\infty}^{+\infty}\frac{u_{xx}^2}{v}dxdt \\
 \displaystyle\leq C+\frac{1}{2}\int_0^T\int_{-\infty}^{+\infty}\frac{u_{xx}^2}{v}dxdt\\
  \displaystyle\qquad+C\int_0^T\int_{-\infty}^{+\infty}\Big(\theta_x^2+\theta^2 v_x^2+\big|\frac{\phi_x}{v}\big|^2\big|\big(\frac{\phi_x}{v}\big)_x\big|^2+u^2_xv_x^2\Big)dxdt\\
 \displaystyle\leq C+\frac{1}{2}\int_0^T\int_{-\infty}^{+\infty}\frac{u_{xx}^2}{v}dxdt+C\int_0^T\int_{-\infty}^{+\infty}\theta_x^2dxdt+C\int_0^T\sup_{x\in\mathbb{R}}\theta^2
 \int_{-\infty}^{+\infty}v_x^2dxdt\\
\displaystyle\ \ +C\sup_{(x,t)\in\mathbb{R}\times[0,T]}\big|\frac{\phi_x}{v}\big|^2\int_0^T\int_{-\infty}^{+\infty}\big|\big(\frac{\phi_x}{v}\big)_x\big|^2dxdt
+C\int_0^T\sup_{x\in\mathbb{R}}u_x^2\int_{-\infty}^{+\infty}v_x^2dxdt\\
 \displaystyle\leq C+\frac{3}{4}\int_0^T\int_{-\infty}^{+\infty}\frac{u_{xx}^2}{v}dxdt
 +C_1\int_0^T\int_{-\infty}^{+\infty}\frac{\theta^\beta\theta_x^2}{v}dxdt,
\end{array}
\right.
\end{equation}
where the following inequality  are used
\begin{eqnarray}\label{u-x L2}
  \int_0^T\sup_{x\in\mathbb{R}} u_x^2dt&\leq& C(\delta)\int_0^T\int_{-\infty}^{+\infty}u_x^2dxdt+\delta\int_0^T\int_{-\infty}^{+\infty}\frac{u_{xx}^2}{v}dxdt\notag\\
  &\leq& C(\delta)+\delta\int_0^T\int_{-\infty}^{+\infty}\frac{u_{xx}^2}{v}dxdt.
\end{eqnarray}
and
\begin{equation}\label{theta-x L2}
\left.\begin{array}{llll}
\displaystyle \sup_{\mathbb{R}} (\theta-2)_+^2&=\displaystyle \sup_{\mathbb{R}}\Big(\int_x^{+\infty}\partial_y(\theta-2)_+(y,t)dy\Big)^2\\
\displaystyle  &\displaystyle\leq \Big(\int_{(\theta>2\bar\theta)(t)}\big|\theta_y\big|dy\Big)^2\\
 \displaystyle &\displaystyle\leq C\int_{-\infty}^{+\infty}\theta_x^2dx.
 \end{array}\right.
\end{equation}
Multiplying \eqref{NSFAC-Lagrange}$_3$ by $(\theta-2)_+$, integrating the resultant over $\mathbb{R}\times(0,T)$ by parts, combining with \eqref{basic energy inequality}, \eqref{integral sup theta} and \eqref{u-x L2}, one has
\begin{align}\label{theta theta-x 1}
 & \frac{1}{2}\int_{-\infty}^{+\infty}(\theta-2)_+^2dx+\int_0^T\int_{(\theta>2)(t)}\frac{\theta^\beta\theta_x^2}{v}dxdt \notag \\
 & \leq C\int_0^T\int_{-\infty}^{+\infty}\theta(\theta-2)_+|u_x|dx+C\int_0^T\int_{-\infty}^{+\infty}\big(u_x^2+\mu^2\big)(\theta-2\bar\theta)_+ dx+C\\
 &\leq C\int_0^T\sup_{\mathbb{R}}\theta\Big(\int_{-\infty}^{+\infty}(\theta-2)_+^2dx+\int_{-\infty}^{+\infty}u_x^2 dx\Big)+C.\notag
 \end{align}
Since
\begin{equation}
\left.\begin{array}{llll}
\displaystyle  \int_0^T\int_{-\infty}^{+\infty}\theta^\beta\theta_x^2dxdt&=&\displaystyle\int_0^T\int_{(\theta>2)(t)}\theta^\beta\theta_x^2dxdt
  +\int_0^T\int_{(\theta\leq2\bar\theta)(t)}\theta^\beta\theta_x^2dxdt \\
   &\leq& \displaystyle C \int_0^T\int_{(\theta>2)(t)}\frac{\theta^\beta\theta_x^2}{v}dxdt
  +C\int_0^T\int_{(\theta\leq2)(t)}\frac{\theta^\beta\theta_x^2}{v\theta^2}dxdt \\
   &\leq&\displaystyle C \int_0^T\int_{(\theta>2)(t)}\frac{\theta^\beta\theta_x^2}{v}dxdt+C,
\end{array}
\right.
\end{equation}
together  \eqref{u-x and u-xx 1} and \eqref{theta theta-x 1}, combining with Gronwall's inequality, one obtains
\begin{equation}\label{tempeture and velocity-1}
  \sup_{t\in[0,T]}\int_{-\infty}^{+\infty}\big((\theta-2)_+^2+u_x^2\big)dx+\int_0^T\int_{-\infty}^{+\infty}\big(\theta^\beta\theta_x^2+ u_{xx}^2\big)dxdt\leq C.
\end{equation}
Rewriting \eqref{NSFAC-Lagrange}$_2$ as
\begin{equation}\label{momentum equation}
  u_t=-\big(\frac{\theta}{v}\big)_x+\frac{u_{xx}}{v}-\frac{u_x v_x}{v^2}-\epsilon\frac{\phi_x}{v}\big(\frac{\phi_x}{v}\big)_x,
\end{equation}
from \eqref{bound of the square modulus of the first derivative}, \eqref{tempeture and velocity-1}, \eqref{the upper and lower bounds of the derivative for phi}, \eqref{u-x L2}, \eqref{bound of the square modulus of the third derivative for phi} and \eqref{tempeture and velocity-1}, we get
\begin{eqnarray}\label{sup u-t L2}
 \int_0^T\int_{-\infty}^{+\infty}u_t^2dxdt\leq C\int_0^T\int_{-\infty}^{+\infty}\Big(u_{xx}^2+u_x^2v_x^2+\theta_x^2+\theta^2v_x^2+\big|\frac{\phi_x}{v}\big|^2\big|\big(\frac{\phi_x}{v}\big)_x\big|^2\Big)dxdt\leq C.
\end{eqnarray}
Together with \eqref{tempeture and velocity-1}, the energy inequality \eqref{energy estimate of velocity} is achieved. The proof of Lemma \ref{velocity-x} is completed.
\end{proof}

\begin{lemma}\label{theta-x}
Let $(v,u,\theta,\phi)$ be a smooth solution of \eqref{NSFAC-Lagrange}-\eqref{initial condition} on $(-\infty,+\infty)\times [0,T]$, then for $\forall(x,t)\in(-\infty,+\infty)\times [0,T]$, the following inequality holds
 \begin{equation}\label{energy estimate of tempeture}
\sup_{0\leq t\leq T}\int_{-\infty}^{+\infty}\theta_{x}^2dx+\int_0^T\int_{-\infty}^{+\infty}\big(\theta_{t}^2+\theta^2_{xx}\big)dxdt\leq C.
 \end{equation}
\end{lemma}
\begin{proof}
Multiplying \eqref{NSFAC-Lagrange}$_5$ by $\theta^\beta\theta_t$ and integrating the resultant over $(0,1)$, by using \eqref{bound of density}, 
\eqref{tempeture and velocity-1}, we have
\begin{align}\label{theta-t theta-xx}
 &\frac{1}{2}\frac{d}{dt}\Big(\int_{-\infty}^{+\infty}\frac{(\theta^\beta\theta_x)^2}{v}dx\Big)+\int_{-\infty}^{+\infty}\theta^\beta\theta_t^2dx\notag \\
  & =-\frac{1}{2}\int_{-\infty}^{+\infty}\frac{(\theta^\beta\theta_x)^2u_x}{v^2}dx+\int_{-\infty}^{+\infty}\frac{\theta^\beta\theta_t\big(-\theta u_x+u_x^2+v^2\mu^2\big)}{v}dx\\
  &\leq C\sup_{x\in\mathbb{R}}|u_x|\theta^{\frac{\beta}{2}}\int_{-\infty}^{+\infty}\theta^{\frac{3\beta}{2}}\theta_x^2dx+\frac{1}{2}\int_{-\infty}^{+\infty}\theta^\beta\theta_t^2dx+C\int_{-\infty}^{+\infty}\theta^{\beta+2}u_x^2dx+C\int_{-\infty}^{+\infty}\theta^\beta\big(u_x^4+\mu^4\big)dx\notag\\
  &\leq C\int_{-\infty}^{+\infty}\theta^{\beta}\theta_x^2dx\int_{-\infty}^{+\infty}(\theta^{\beta}\theta_x)^2dx+\frac{1}{2}\int_{-\infty}^{+\infty}\theta^{\beta}\theta_t^2dx+C\sup_{x\in\mathbb{R}}\big((\theta-1)^{2\beta+2}+u_x^4+\mu^4\big)+C.\notag
\end{align}
Now we deal with the term $\sup_{x\in\mathbb{R}}\big((\theta-1)^{2\beta+2}+u_x^4+\mu^4\big)$ in the last inequality of \eqref{theta-t theta-xx}.
Applying Lemma \ref{velocity-x}, direct computation shows that
\begin{eqnarray}
  \int_0^T\sup_{x\in\mathbb{R}}u_x^4dt  &\leq& C\int_0^T\int_{-\infty}^{+\infty}|u_x^3u_{xx}|dxdt\notag\\
  &\leq&C\int_0^T\sup_{x\in\mathbb{R}}u_x^2\Big(\int_{-\infty}^{+\infty}u_x^2dx\Big)^\frac{1}{2}\Big(\int_{-\infty}^{+\infty}u_{xx}^2dx\Big)^\frac{1}{2}dt\\
   &\leq&\frac12\int_0^T\sup_{x\in\mathbb{R}}u_x^4dt+C\int_0^T\int_{-\infty}^{+\infty}\big(u_x^2+u_{xx}^2)dxdt\notag\\
      &\leq&\frac12\int_0^T\sup_{x\in\mathbb{R}}u_x^4dt+C,\notag
\end{eqnarray}
then
\begin{equation}\label{u-x^4}
 \int_0^T\sup_{x\in\mathbb{R}}u_x^4dt\leq C.
\end{equation}
Combining with \eqref{bound of the square modulus of the third derivative for phi}, by the same way above, we  have
\begin{equation}\label{mu^4}
 \int_0^T\sup_{x\in\mathbb{R}}\mu^4dt\leq C.
\end{equation}
Moreover, by using Sobolev embedding theorem and \eqref{tempeture and velocity-1}, we get
\begin{align}\label{the upper bound estimate for theta}
  \sup_{x\in\mathbb{R}}(\theta-1)^{2\beta+2}
  &\leq C(\delta)\int_{-\infty}^{+\infty}(\theta-1)^{2\beta+2}dx+\delta\int_{-\infty}^{+\infty}(\theta-1)^{2\beta}\theta_x^2dx\notag\\
  &\leq \frac{1}{2} \sup_{x\in\mathbb{R}}(\theta-1)^{2\beta+2}+C(\delta)+C\delta\int_{-\infty}^{+\infty}(\theta^\beta\theta_x)^2dx.
\end{align}
Substituting \eqref{u-x^4}, \eqref{mu^4}, \eqref{the upper bound estimate for theta} into \eqref{theta-t theta-xx}, by using Gronwall's inequality, we obtain
\begin{equation}\label{derivative estimation for theta}
\sup_{0\leq t\leq T}\int_{-\infty}^{+\infty}(\theta^\beta\theta_x)^2dx+\int_0^T\int_{-\infty}^{+\infty}\theta^\beta\theta_t^2dxdt\leq C.
\end{equation}
Therefore,  in view of \eqref{the upper bound estimate for theta}, we have
\begin{equation}\label{the upper bound estimate for temperature}
 \sup_{(x,t)\in(-\infty,+\infty)\times[0,T]}\theta\leq C.
\end{equation}
Thus, both \eqref{the upper bound estimate for theta} and \eqref{derivative estimation for theta} lead to
\begin{align}
\displaystyle  \sup_{0 \le t\le T}\int_{-\infty}^{+\infty}\theta_{x}^2 dx+\int_0^T\int_{-\infty}^{+\infty} \theta_t^2dxdt\le C.
\end{align}

From \eqref{NSFAC-Lagrange}$_5$ again, we also have
\begin{equation}\label{energy conservation equation theta-xx}
\frac{\theta^\beta\theta_{xx}}{v}=\theta_t- \frac{\beta\theta^{\beta-1}\theta_x^2}{v}+\frac{\theta^\beta\theta_x v_x}{v^2}+\frac{\theta u_x}{v}-\frac{u_x^2+(v\mu)^2}{v},
\end{equation}
which yields that
\begin{equation}\label{the estimate of tempture-xx L2}
\left.\begin{array}{llll}
\displaystyle\int_0^T\int_{-\infty}^{+\infty}\theta_{xx}^2dxdt&\displaystyle\leq C \int_0^T\int_{-\infty}^{+\infty}\big(\theta_x^2v_x^2+\theta_x^4+u_x^4+\mu^4+u_x^2+\theta_t^2\big)dxdt\\
 &\leq\displaystyle C(\delta)+C\delta\int_0^T\sup_{x\in\mathbb{R}}\theta_x^2dt\\
  &\leq\displaystyle C(\delta)+C\delta\int_0^T\int_{-\infty}^{+\infty}\theta_{xx}^2dxdt.
\end{array}
\right.
\end{equation}
Furthermore,  by using maximum principle, we obtain $-1\leq\phi\leq1$.
The proof of Lemma \ref{theta-x} is completed.
\end{proof}

So far, from the a priori estimates of solutions (see Lemma \ref{Local existence}-Lemma \ref{theta-x}), Theorem \ref{thm-global} can be obtained by extending the local solutions globally in time. For details, please refer to  \cite{DLL2013}, \cite{HSS2020}, 
 and the references therein.

\

\noindent{\bf Acknowledgments.} The authors would like to thank the anonymous referees for their careful comments and suggestions leading to  improvements in the paper.

\end{document}